\documentclass[preprint]{imsart}

\usepackage{natbib}
\usepackage{hyperref}
\usepackage{amsmath}
\usepackage{graphicx}
\usepackage{mathrsfs}
\usepackage{amsthm}
\usepackage{amscd}
\usepackage{amssymb}
\usepackage{latexsym}
\usepackage{stmaryrd}
\usepackage{wasysym}
\usepackage{enumitem}
\usepackage{verbatim}
\usepackage{nicefrac}

\startlocaldefs
\numberwithin{equation}{section}
\theoremstyle{plain}
\endlocaldefs

\newcommand{\hyperboloid}[1]{\mathbb{H}_{#1}}

\newcommand{\iid}{iid} 
\newcommand{\isometries}{\mathop{isom}}

\newcommand{\Prob}[1]{\mathop{Pr}\left(#1\right)}

\newcommand{\graph}{\mathrm{graph}}

\newcommand{\Expect}[1]{\mathbb{E}\left[ #1 \right] }
\newcommand{\Loglike}{\ell}
\newcommand{\NormalizedLoglike}{\ell_{norm}}
\newcommand{\linkfn}{w_n}
\newcommand{\linkfnseq}{w_{1:\infty}}
\newcommand{\CrossEntropy}{\overline{\Loglike}}
\newcommand{\NormalizedCrossEntropy}{\overline{\Loglike}_{norm}}
\newcommand{\LikeFnClass}{\mathcal{L}_n}
\newcommand{\convprob}{\ensuremath{\stackrel{P}{\rightarrow}}}
\newcommand{\eqdist}{\ensuremath{\stackrel{d}{=}}}

\newcommand{\truedyaddist}{\pi^*}
\newcommand{\logitbound}{v_n}
\newcommand{\MLE}{\hat{x}_{1:n}}
\newcommand{\trueparam}{x^*_{1:n}}
\DeclareMathOperator*{\argmax}{argmax} 
\DeclareMathOperator*{\logit}{logit}
\DeclareMathOperator*{\arccosh}{arccosh}
\newcommand{\IsometryClass}[1]{ \left[ #1 \right] }
\newcommand{\randomestimatedcoordinates}{\hat{X}_{1:n}}
\newcommand{\estimatedcoordinates}{\hat{x}_{1:n}}

\newcommand{\MaxDeviation}{\Gamma_n}
\newcommand{\KLDiv}{D}

\newcommand{\LatentSpace}{M}
\newcommand{\ContinuousSpace}{M}
\newcommand{\MetricSymbol}{\mathop{dist}}
\newcommand{\Metric}[2]{\MetricSymbol\left(#1,#2\right)}
\newcommand{\MetricSpace}{(\ContinuousSpace,\MetricSymbol)}

\newcommand{\QuasiUniformDensity}{q}

\newcommand{\HyperbolicGaussianParam}{\sigma}

\newcommand{\Reals}{\mathbb{R}} 

\newcommand{\ra}{\rightarrow}
\newcommand{\HH}{\mathbb{H}_2} 
\newcommand{\disteq}{\eqdist}  
\newcommand{\adjacencymatrix}{G}

\newcommand{\smooth}{smooth}
\newcommand{\logitbounded}{logit-bounded}
\newcommand{\rigid}{rigid}
\newcommand{\regular}{regular}

\theoremstyle{definition}
\newtheorem{thm}{Theorem}
\newtheorem{defn}[thm]{Definition}
\newtheorem{cor}[thm]{Corollary}
\newtheorem{lem}[thm]{Lemma}
\newtheorem{prop}[thm]{Proposition}

\begin{document}
\begin{frontmatter}
\title{Consistency of Maximum Likelihood for Continuous-Space Network Models I}
\runtitle{Consistency for Continuous Network Models I}

\begin{aug}
\author{Cosma Rohilla Shalizi}
\address{Department of Statistics and Machine Learning\\ Carnegie Mellon University\\ Pittsburgh, PA\\ 15213 USA\\
{\em and } Santa Fe Institute\\ 1399 Hyde Park Road\\ Santa Fe, NM\\ 87501, USA\\
\href{Url}{cshalizi@cmu.edu}
}

\and
\author{Dena Marie Asta}
\address{Department of Statistics\\ Ohio State University\\ Columbus, OH\\ 43210 USA\\
\href{Url}{dasta@stat.osu.edu}
}
\runauthor{Shalizi and Asta}
\end{aug}

\begin{abstract}
  A very popular class of models for networks posits that each node is
  represented by a point in a continuous latent space, and that the probability
  of an edge between nodes is a decreasing function of the distance between
  them in this latent space.  We study the {\em embedding problem} for these
  models, of recovering the latent positions from the observed graph.  Assuming
  certain natural symmetry and smoothness properties, we establish the uniform
  convergence of the log-likelihood of latent positions as the number of nodes
  grows.  A consequence is that the maximum likelihood embedding converges on
  the true positions in a certain information-theoretic sense.  Extensions of
  these results, to recovering distributions in the latent space, and so
  distributions over arbitrarily large graphs, will be treated in the sequel.
\end{abstract}




\end{frontmatter}

\section{Introduction}
The statistical analysis of network data, like other sorts of statistical
analysis, models the data we observe as the outcome of stochastic processes,
and rests on inferring aspects of those processes from their results.  It is
essential that the methods of inference be consistent, that as they get more
and more information, they should come closer and closer to the truth.  In this
paper, we address the consistency of non-parametric maximum likelihood
estimation for a popular class of network models, those based on continuous
latent spaces.

In these models, every node in the network corresponds to a point in a latent,
continuous metric space, and the probability of an edge or tie between two
nodes is a decreasing function of the distance between their points in the
latent space.  These models are popular because they are easily interpreted in
very plausible ways, and often provide good fits to data.  Moreover, they have
extremely convenient mathematical and statistical properties: they lead to
exchangeable, projectively-consistent distributions over graphs; the comparison
of two networks reduces to comparing two clouds of points in the latent space,
or even to comparing two densities therein; it is easy to simulate new networks
from the estimated model for purposes of bootstrapping, etc.  While the latent
space has typically been taken to be a low-dimensional Euclidean space
\citep{Hoff-Raftery-Handcock}, recent work has suggested that in many
applications it would be better to take the space to non-Euclidean,
specifically negatively curved or hyperbolic
\citep{Krioukov-et-al-hyperbolic-geometry,
Asta-CRS-geometric-network-comparison}.

We can estimate continuous latent space models in the sense of an {\em
embedding}: given an observed graph, we wish to work backwards the locations of
the nodes in the latent space, i.e., to ``embed'' the graph in the latent
space.  The most straightforward method of embedding is a \textit{maximum
likelihood estimator} (MLE), treating the latent position of each node as a
parameter (or vector of parameters).  While it is straightforward to {\em say}
that a good embedding should converge on the true coordinates as the number of
nodes $n\rightarrow\infty$, making this mathematically precise is somewhat
tricky.  (We would, for example, need to define a metric on the space of
embeddings of graphs of different sizes.)  Instead, we prove the next best
thing: that for continuous latent space models of sufficient symmetry and
tameness, the distribution over graphs implied by the MLE converges, in
normalized Kullback-Leibler (KL) divergence, to the distribution implied by the
true embedding.  That is, the MLE becomes statistically indistinguishable, in
its observable consequences, from the truth.  This is a consequence of a result
we establish along the way, about the uniform convergence of normalized
log-likelihoods to their expectation values (Theorem
\ref{thm:uniform-concentration-of-likelihood}); the rate of this uniform
convergence upper-bounds the rate at which the MLE approaches the true
embedding in KL divergence.

In the sequel in preparation, we combine our results about normalized
log-likelihood with the construction of a specific class of metrics on growing
sequences of embeddings, to establish a more conventional, coordinate-wise
notion of consistency, and consistency for a subsequent estimator of the node
density in the latent space.

Section \S \ref{sec:background} reviews background on continuous latent space
models of networks.  Section \S \ref{sec:geo-inference} states our main
results, along with certain technical assumptions.  All proofs, and a number of
subsidiary results and lemmas, are deferred to Section \S\ref{sec:proofs}.  In
the Appendix, we note that our results generalize to mis-specified models.

\section{Background}
\label{sec:background}
In many, though not all, network data-analysis situations, we have only one
network --- perhaps not even all of that one network --- from which we
nonetheless want to draw inferences about the whole data-generating process.
This clearly will require a law of large numbers or ergodic theorem to ensure
that a single large sample is representative of the whole process.  The
network, however, is a single high-dimensional object whose every part is
dependent on every other part.  This is also true of time-series and spatial
data, but there we can often use the fact that distant parts of the data should
be nearly independent of each other.  While general networks often exhibit such
decay, networks in nature often lack a natural, exogenous sense of distance (in
the technical, geometric sense) that explains such decay.

{\bf Continuous latent space} (CLS) models are precisely generative models for
networks which exhibit just such an exogenous sense of distance.  Each node is
represented as a location in a continuous metric space, the {\bf latent space}.
Conditional on the vector of all node locations, the probability of an edge
between two nodes is a decreasing function of the distance between their
locations, and all edges are independent.  Generative models for networks for
which the existence of different edges is conditionally independent with
respect to some latent quantity $\mu$ are common; however CLS models, at least
as taken in this paper, are distinguished by the particular geometric form that
$\mu$ takes.

As mentioned above, the best-known CLS model for social networks is that of
\citet{Hoff-Raftery-Handcock}, where the metric space is taken to be Euclidean,
and node locations are assumed to be drawn \iid{}ly from a Gaussian
distribution.  In random geometric graphs
\citep{Penrose-random-geometric-graphs}, the locations are drawn \iid{}ly from
a distribution on a metric space possibly more general than Euclidean space and
the probabilities of connecting edges are either $0$ or $1$ based on a
threshold.

As also mentioned above, there is more recent work which indicates that for
some applications it would be better to let the latent space be negatively
curved, i.e. hyperbolic \citep{Albert-et-al-negative-curvature-of-networks,
  Kennedy-et-al-hyperbolicity-of-networks, Krioukov-et-al-hyperbolic-geometry}.
Mathematically, this is because many real networks can be naturally embedded
into such spaces.  More substantively, many real-world networks show highly
skewed degree distributions, very short path lengths, a division into a core
and peripheries where short paths between peripheral nodes ``bend back''
towards the core, and a hierarchical organization of clustering.  Thus if the
latent space is chosen to be a certain hyperboloid, one naturally obtains
graphs exhibiting all these properties
\citep{Krioukov-et-al-hyperbolic-geometry}.

The CLS models we have mentioned so far have presumed that node locations
follow tractable, parametric families in the latent space.  This is
mathematically inessential --- many of the results carry over perfectly well to
arbitrary densities --- and scientifically unmotivated.  Because CLS models may
need very different spaces depending on applications, we investigate
consistency of nonparametric estimation for them at a level of generality which
abstracts away from many of the details of particular spaces and their metrics.

To the best of our knowledge, there are no results in the existing literature
on the consistency of embedding for CLS models where edge probabilties vary
continuously with distance\footnote{Computationally-tractable and consistent
  embedding algorithms exist for some kinds of random geometric graph where
  edges are deterministically present between sufficiently-close nodes and
  otherwise deterministically absent
  \citep{Dani-et-al-recontruction-of-random-geometric-graphs}, but they rely
  crucially on deterministic links, and their statistical efficiency is
  unknown.  Uniform consistency for variants of these sorts of CLS models, such
  as \textit{random dot product graphs} (RDPG) \cite{young2007random}, have
  been well established (c.f. \cite{athreya2017statistical}); the inherently
  linear algebraic methods used to develop estimators and consistency results
  in the RDPG setting do not seem portable in the metric setting.}.

\section{Geometric Network Inference}
\label{sec:geo-inference}
Our goal is to show that when the continuous latent space model is sufficiently
smooth, and the geometry of the latent space is itself sufficiently symmetric,
then the maximum-likelihood embedding of a graph converges, in normalized
Kullback-Leibler divergence, to the true locations of the nodes (Theorem
\ref{thm:uniform-concentration-of-likelihood}).  As an intermediate step, we
show the uniform convergence of normalized log-likelihoods on their expectation
values, at an explicit rate, which also gives us the rate of KL convergence of
the MLE on the truth.  All proofs are postponed to \S \ref{sec:proofs}.

\subsection{Setting and Conventions}

We consider only simple, undirected, unlabeled graphs; we will write a random
graph as $G$, and will sometimes abuse notation to also write $G$ for the
adjacency matrix, so that $G_{pq}=G_{qp}=1$ if there is an edge between nodes
$p$ and $q$, and $=0$ otherwise.

All random graphs $G$ in this paper have conditionally independent edges; that
is, we assume for each $G$, there exists a random quantity $\mu$ such that
$G\mid\mu$ has independently distributed edges.  A \textit{continuous latent
  space model} assumes that $\mu$ has a certain geometric nature, which will be
defined in the succeeding paragraphs.

All the metrics of metric spaces will be denoted by $\MetricSymbol$; context
will make clear which metric $\MetricSymbol$ is describing.  Our model for
generating random graphs begins with a metric measure space $\ContinuousSpace$,
a metric space equipped with a Borel measure, and the corresponding group
$\isometries(\ContinuousSpace)$ of measure-preserving isometries
$\ContinuousSpace\cong\ContinuousSpace$.  Every node is located at
(equivalently, ``represented by'' or ``labeled with'') a point in
$\ContinuousSpace$, $x_i$ for the $i^{\mathrm{th}}$ node; the location of the
first $n$ nodes is $x_{1:n} \in \ContinuousSpace^n$, and a countable sequence
of locations will be $x_{1:\infty}$.  For each $n$, there is a non-increasing
{\bf link function} $w_n: [0,\infty) \mapsto [0,1]$, and nodes $i$ and $j$ are
joined by an edge with probability $w_n(\Metric{x_i}{x_j})$.  By a
\textit{latent space} $(\ContinuousSpace,\linkfnseq)$, we will mean the
combination of $\ContinuousSpace$ and a sequence $\linkfnseq$ of link functions
$w_1,w_2,\ldots$.  When the latent space is understood, we write
$\graph_n(x_{1:n})$ for the distribution of a random graph on $n$ vertices
located at $x_{1:n}$.  Thus in the particular case $G=\graph_n(x_{1:n})$,
we have $\mu=x_{1:n}$.

It is clear that for any $\phi \in \isometries(\ContinuousSpace)$, we have
for every $n$,
\begin{equation}
\graph_n(x_{1:n})\eqdist\graph_n(\phi(x_{1:n}))
\end{equation}
Accordingly, we will use $\IsometryClass{x_{1:n}}$ to indicate the equivalence
class of $n$-tuples in $\ContinuousSpace^n$ carried by isometries to $x_{1:n}$;
the metric on $\ContinuousSpace$ extends to these isometry classes in the
natural way,
\begin{equation}
  \Metric{\IsometryClass{x_{1:n}}}{\IsometryClass{y_{1:n}}}=\inf_{\phi\in\isometries(\ContinuousSpace)}{\sum_{i=1}^n\Metric{x_i}{\phi(y_i)}} ~. \label{eqn:metric-for-isometry-classes}
\end{equation}
We cannot hope to find $x_{1:n}$ by observing the graph it leads to, but we
can hope to identify $\IsometryClass{x_{1:n}}$.

\paragraph{Conventions} When $n$ and $m$ are integers, $n<m$, $n:m$ will be the
set $\{n, n+1, \ldots m-1, m \}$.  Unless otherwise specified, all limits will
be taken as $n\rightarrow\infty$.  All probabilities and expectations will be
taken with respect to the actual generating distribution of $G$.

\subsection{Axioms on the generative model}

We recall that a metric space $\ContinuousSpace$ is \textit{$k$-homogeneous} if
every isometry between finite submetric spaces each of size $k$ of
$\ContinuousSpace$ extends to an isometry on $\ContinuousSpace$, an isometry
$\ContinuousSpace\ra\ContinuousSpace$.  There we call a metric space
$\ContinuousSpace$ \textit{$\infty$-homogeneous} if every isometry between
finite submetric spaces of $\ContinuousSpace$ extends to an isometry on
$\ContinuousSpace$.  The literature takes \textit{homogeneous} to usually mean
$1$-homogeneous but to sometimes to mean $\infty$-homogeneous.  Motivating
examples are Euclidean space $\Reals^d$ and the \textit{Poincar\'{e} Halfplane}
$\HH$, described in \S\ref{subsec:example}.  Almost any example of a metric
space with a single ``singularity'' $x_1$, such as a ``figure 8,'' is not
$1$-homogeneous; for a close enough point $x_2$, there are also points
$x_3,x_4$ such that $\MetricSymbol(x_1,x_2)=\MetricSymbol(x_3,x_4)$, but
intuitively there cannot be any isometry carrying a singularity to a
non-singularity.  An example of a metric space that is $1$-homogeneous but not
$\infty$-homogeneous is the orientable surface of infinite genus.



%

Identifiability of graph distributions determined by certain CLS models is
possible.  We define such CLS models below.

\begin{defn}
  A latent space $(\ContinuousSpace, \linkfnseq)$ is {\bf \regular{}} when:
  \begin{enumerate}
    \item $\ContinuousSpace$ is a complete $\infty$-homogeneous Riemannian manifold;
	\item The group of isometries on $\ContinuousSpace$ has only finitely-many connected components;
    \item The function $\linkfn$ is injective and smooth for each $n$; and 
	\item The sequence $\linkfnseq$ satisifes $-\logitbound\leqslant\logit{\linkfn(\Metric{x}{y})}\leqslant\logitbound$ for some $\logitbound\in o(\sqrt{n})$. 
  \end{enumerate}
  \label{defn:regular}
\end{defn}

\begin{prop}
\label{prop:reasonable.symmetries}
The metric spaces $\Reals^d$ and $\HH$ satisfy points (1) and (2) of Definition \ref{defn:regular} with
\begin{equation*}
  B_{\HH}=B_{\Reals^d}=2.
\end{equation*}
where $B_{\ContinuousSpace}$ denotes the number of connected components of the group of isometries on a metric space $\ContinuousSpace$.  
\end{prop}

Demanding that $\logitbound = o(\sqrt{n})$ is done with an eye towards the
needs of the proofs in \S\ref{sec:proofs}.  Some common examples of link
functions (cf. \citep{Krioukov-et-al-hyperbolic-geometry}) include the
following two kinds:
\begin{equation}
	\label{eqn:example-link-functions}
  \linkfn(t)=\begin{cases}1&t\leqslant\ln n\\0&t>\ln n\end{cases}\quad
	\linkfn(t)=\frac{1}{1+e^{(\nicefrac{T_n^{-1}}{2})(t-\ln n)}}
\end{equation}

The first sort defines a graph where edges are deterministically present between sufficiently-close nodes, and deterministically absent between more distant nodes.
The second sort, in
which the sequence of $T_n$s are fixed \textit{temperature} parameters; the higher
the temperature $T_n$, the closer the link function is to a constant
probability $\nicefrac{1}{2}$.  The determinism of the first kind violates
logit-boundedness.  The second kind satisfies logit-boundedness when
$T_n\in o((\ln n)n^{-\nicefrac{1}{2}})$ and $t\leqslant\ln n$; in that
case
\[
\logit{\frac{1}{1+e^{{(\nicefrac{T_n^{-1}}{2})(t-\ln n)}}}}=-{(\nicefrac{T_n^{-1}}{2})(t-\ln n)}\in o(\sqrt{n})
\]

By extension, a CLS model is \regular{} when $(\ContinuousSpace, \linkfnseq)$
is.  The proof of the following proposition, a straightforward consequence of
$\infty$-homogeneity and injectivity of the link functions, is omitted.

\begin{prop}
  \label{thm:equivalent.random.graphs}
  For \regular{} CLS model
  \begin{equation}
    \graph_n(x_{1:n})\eqdist\graph_n(y_{1:n}) \iff \IsometryClass{x_{1:n}}=\IsometryClass{y_{1:n}}\quad n=1,2,\ldots
  \end{equation}
\end{prop}

Theorem \ref{thm:equivalent.random.graphs} lets us identify graph distributions
of the form $\graph_n(x_{1:n})$ with isometry classes
$\IsometryClass{x_{1:n}}$.  


\subsection{An example in the literature}\label{subsec:example}
Latent spaces of the form
\begin{equation*}
  (\HH,\linkfn).
\end{equation*}
where $\mathbb{H}_2=\{z\in\mathbb{C}\;|\;\mathop{Im}(z)>0\}$ is the \textit{Poincar\'{e} halfplane} with metric
$$dz=y^{-2}\;dx\;dy$$
were introduced \citep{Krioukov-et-al-hyperbolic-geometry} to model networks in
nature with tree-like characteristics (e.g. the internet).  With the first link
function defined in (\ref{eqn:example-link-functions}), regularity is violated
in multiple ways; the link functions are not logit-bounded as noted earlier,
but also the link functions are neither smooth nor injective.  With the second
set of $\linkfn$'s in (\ref{eqn:example-link-functions}), with temperature
parameters $T_n\in o(\sqrt{n}^{-1})$, the CLS model is \regular{}.  Such CLS
models have been shown to model salient for both large-scale and small-scale
properties of various sorts of social networks
\cite{smith2017geometry}, building on work of
\citet{Krioukov-et-al-hyperbolic-geometry}.  

\subsection{Estimation}

Given a latent space model $(\ContinuousSpace, \linkfnseq)$ and an $n$-node
graph $G$, the likelihood $\mathcal{L}(x_{1:n};G)$ of observing coordinates $x_{1:n}\in\ContinuousSpace^m$ is given as the product of edge probabilities\footnote{In this expression, every dyad appears twice, once as $p, q)$ and again as $(q,p)$, but, since $G$ is undirected, contributing the same factor to the likelihood each time.  This is thus the square of another possible likelihood which counted each dyad only once.  This will make no difference to the analysis, apart from needing to track a factor of 2 through our results.}:
\[
\mathcal{L}(x_{1:n};G) \equiv \prod_{p=1}^{n}{\prod_{q\neq p}{w_n(\MetricSymbol{x_p}{x_q})^{G_{pq}} (1-w_n(\MetricSymbol{x_p}{x_q}))^{1-G_{pq}}}}
\]

A maximum likelihood (ML) {\em embedding} of an $n$-node graph $G$ into
$\ContinuousSpace$ is
\begin{equation}
  \estimatedcoordinates=\argmax_{x_{1:n}\in\ContinuousSpace^n}{\mathcal{L}(x_{1:n};G)}
\end{equation}

Taking logs and dividing by the number of summands, we obtain the \textit{normalized log-likelihood} $\Loglike(x_{1:n};G)$ of observing coordinates $x_{1:n}\in\ContinuousSpace^m$ by:
\begin{equation}
  \label{eqn:likelihood.function}
	\Loglike(x_{1:n};G)= 
	\frac{1}{n(n-1)}\left( \sum_{(p,q)\in G}{\log{(\linkfn(\Metric{x_p}{x_q}))}}+\sum_{(p,q)\not\in G}{\log{(1-\linkfn(\Metric{x_p}{x_q}))}} \right)
\end{equation}
As usual, when there is no
ambiguity about the graph $G$ providing the data, we will suppress that as an
argument, writing $\Loglike(x_{1:n})$.

Taking expectations with respect to the actual graph
distribution of a random graph $G$ having $n$ nodes, we define the expected
normalized log-likelihood by
\begin{equation}
  \label{eqn:cross-entropy}
	\CrossEntropy(x_{1:n})=\mathbb{E}_G[\Loglike(x_{1:n}; G)],
\end{equation}
As we review in Sec.\ \ref{ref:information-theory}, well-known results from
information theory show that $-\CrossEntropy(x_{1:n})$ can always be decomposed
into the sum of two non-negative terms, $-\CrossEntropy(x_{1:n}) = H +
D(x_{1:n})$.  Here the first term, the ``source entropy rate'' $H$, captures
the inherent stochasticity of the data source.  The second term, the
``divergence rate'' $D(x_{1:n})$, measures the distance, or rather the
normalized Kullback-Leibler divergence, between that true distribution and
$\graph_n(x_{1:n})$.  Among other properties, $D(x_{1:n})$ controls the power
of any hypothesis test to distinguish $\graph_n(x_{1:n})$ from the true
distribution.  This divergence is minimized by $x_{1:n}=\trueparam$;
when the model is well-specified, $D(\trueparam)=0$.

We are now in a position to state our main results.  

\begin{thm}
  Suppose that the CLS model is \regular{}.  Then
  \begin{equation*}
    \MaxDeviation \equiv \sup_{x_{1:n}}{|\Loglike(x_{1:n})-\CrossEntropy(x_{1:n})|} \convprob 0
  \end{equation*}
  \label{thm:uniform-concentration-of-likelihood}
\end{thm}

\begin{cor}
  Suppose that the CLS model is \regular{}, and $G\sim\graph_n{\trueparam}$.  Then
  \begin{equation*}
   D(\MLE) - D(\trueparam) \leq 2\MaxDeviation \convprob 0
  \end{equation*}
  \label{cor:uniform-concentration-of-likelihood}
\end{cor}

\section{Proofs}
\label{sec:proofs}
This section furnishes proofs of main results about networks.  We can sketch the general
approach as follows.  We show that the expected log-likelihood achives its
maximum precisely at the true coordinates up to isometry (Lemma
\ref{lemma:crossent-minimized-on-truth}).  We then show that (in large graphs)
the log-likelihood $\Loglike(x_{1:n})$ is, with arbitrarily high probability,
arbitrarily close to its expectation value for each $x_{1:n}$ (Lemmas
\ref{lemma:individual-concentration-at-individual-rate} and
\ref{lem:individual-concentration-at-uniform-rate}).  We then extend that to a
uniform convergence in probability, over all of $\ContinuousSpace^n$ (Theorem
\ref{thm:uniform-concentration-of-likelihood}).  To do so, we need to bound the
richness (\textit{pseudo-dimension}
\cite[\S11]{Anthony-Bartlett-neural-network-learning}, a continuous
generalization of VC dimension) of the family of log-likelihood functions
(Theorem \ref{thm:pseudo-dimension-of-likelihood-functions}), which involves
the complexity of the latent space's geometry, specifically of its isometry
group $\isometries(\ContinuousSpace)$.  Having done this, we have shown that
the MLE also has close to the maximum expected log-likelihood.  We emphasize
this because the expected log-likelihood has a natural information-theoretic
interpretation in terms of divergence from the truth (Eq.\ \ref{eqn:crossent-is-ent-plus-kl} below).

\subsection{Notation}

Before we dive into details, we first introduce some additional notation for
our proofs.  We will use $G$ for both a (random or deterministic) graph and its
adjacency matrix.

We fix the latent space as $(\ContinuousSpace, \linkfnseq)$.  For brevity,
define
\begin{equation}
\lambda_n(x_p,x_q) \equiv \logit{\linkfn(\Metric{x_p}{x_q})} ~.
\end{equation}
As usual with binary observations, we can rewrite
(\ref{eqn:likelihood.function}) so that the sum is taken over all pairs of
distinct $(p,q)$ and then replace each summand by
$\log{(1-\linkfn(\Metric{x_p}{x_q}))}+G_{pq}\lambda_n(x_p,x_q)$.  This brings
out that the only data-dependent (and hence random) part of $\Loglike$ is
linear in the entries of the adjacency matrix, and in the logit transform of
the link-probability function.  As usual, when there is no ambiguity about the
graph $G$ providing the data, we will suppress that as an argument, writing
$\Loglike(x_{1:n})$.  We write the class of log-likelihood functions as
$\LikeFnClass$.  

\subsection{Information Theory}
\label{ref:information-theory}

Recall the definition of expected normalized log-likelihood from
Eq.\ \eqref{eqn:cross-entropy}:

Taking expectations with respect to the actual graph
distribution of a random graph $G$ having $n$ nodes, we define the expected
normalized log-likelihood (the {\em cross-entropy}; \citealt[ch.\
2]{Cover-and-Thomas-2nd}) by
\begin{equation}
  \label{eqn:cross-entropy}
	\CrossEntropy(x_{1:n})=\mathbb{E}_G[\Loglike(x_{1:n}; G)],
\end{equation}

where the expectation is taken with respect to the random graph $G$ (and not
the random graph $G$ conditioned on some random equantity $\mu$ making the
edges independent.)  For notational convenience, set
\begin{align*}
  \pi_{pq}(a)&=\Prob{\adjacencymatrix{}_{pq}=a \mid x_p,x_q}\\
	\truedyaddist_{pq}(a)&=\Prob{\adjacencymatrix{}_{pq}=a \mid x_p^*,x_q^*}
\end{align*}
(so that $\pi_{pq}(1)=w_n(x_p,x_q)$ and  $\pi^*_{pq}(1)=w_n(x^*_p,x^*_q)$.)  Then
\begin{equation*}
   \CrossEntropy(x_{1:n})
  =\frac{1}{n(n-1)} \sum_{1\leq p<q\leq n}{
	{\sum_{a \in \left\{0,1\right\}}{\truedyaddist_{pq}(a) \log{\pi_{pq}(a)}}}}.
\end{equation*}

In information theory \citep[ch.\ 2]{Cover-and-Thomas-2nd}, this quantity
is known as the (normalized) {\bf cross-entropy}, and we know that
\begin{equation*}
-\sum_{a \in \left\{0,1\right\}}{\truedyaddist_{pq}(a) \log{\pi_{pq}(a)}}=H[\truedyaddist_{pq}] + \KLDiv(\truedyaddist_{pq}\lVert\pi_{pq}),
\end{equation*}
as the left side is the cross-entropy of the distribution $\pi_{pq}$ with
respect to the distribution $\truedyaddist_{pq}$ and the right side is the sum
of ordinary entropy $H$ with the Kullback-Leibler divergence $D$.  Since both
entropy and KL divergence are additive over independent random variables
\citep[ch.\ 2]{Cover-and-Thomas-2nd} like $\adjacencymatrix{}_{pq}$, we
have\footnote{The decomposition of expected log-likelihood into a entropy term
  which only involves the true distribution of the data, plus a KL divergence,
  goes back to at least \citet{Kullback-info-theory-and-stats}.}, defining
$H[\truedyaddist]$ and $\KLDiv(\truedyaddist \lVert \pi)$ in the obvious ways,
\begin{equation}
  -\CrossEntropy(x_{1:n}) = H[\truedyaddist] + \KLDiv(\truedyaddist \lVert \pi)
  \label{eqn:crossent-is-ent-plus-kl}
\end{equation}

Unsurprisingly, $\CrossEntropy$ achieves a maximum at the (isometry class of)
the true coordinates\footnote{The statement and proof of the following lemma
presume that the model is well-specified.  If the model is mis-specified,
then $\inf_{x_{1:n}}{\KLDiv(\truedyaddist\lVert\pi)}$ is still well-defined, and
still defines the value of the supremum for $\CrossEntropy$.  The pseudo-true
parameter value would be one which actually attained the infimum of the
divergence \citep{White-specification-analysis}.  This, in turn, would be the
projection of $\truedyaddist$ on to the manifold of distributions generated
by the model \citep{Diff-geo-in-stat-inf}.  All later invocations of Lemma
\ref{lemma:crossent-minimized-on-truth} could be replaced by the assumption
merely that this pseudo-truth is well-defined.}.

\begin{lem}
  For $\infty$-homogeneous $\ContinuousSpace$ and
  $G\sim\graph_n(\trueparam)$,
  \begin{equation*}
    \IsometryClass{\trueparam}=\argmax_{x_{1:n} \in \ContinuousSpace^n}{\NormalizedCrossEntropy(x_{1:n})} ~.
  \end{equation*}
  \label{lemma:crossent-minimized-on-truth}
\end{lem}
\begin{proof}
  Letting $H$ and $D$ respectively denote entropy and KL divergence as in (\ref{eqn:crossent-is-ent-plus-kl}),
  $\KLDiv(\truedyaddist \lVert \pi) \geq 0$, with equality if and only if
  $\truedyaddist = \pi$.  Therefore we have that the divergence-minimizing
  $\pi$ must be the distribution over graphs generated by some $x_{1:n} \in
  \IsometryClass{\trueparam}$, and conversely that any parameter vector in that
  isometry class will minimize the divergence.  The lemma follows from
  (\ref{eqn:crossent-is-ent-plus-kl}).
\end{proof}

\subsection{Geometric Complexity of Continuous Spaces}
\label{sec:dimension-bound-details}

For various adjacency matrices $\adjacencymatrix^1, \adjacencymatrix^2$, etc.,
let us abbreviate $\Loglike(x_{1:n}; \adjacencymatrix^i)$ as
$\Loglike^i(x_{1:n})$ (following \citealt[p.\
91]{Anthony-Bartlett-neural-network-learning}).  Let us pick $r$ different
adjacency matrices $\adjacencymatrix^1, \ldots, \adjacencymatrix^r$, and set
$\psi(x_{1:n}) = \left(\Loglike^1(x_{1:n}), \ldots,
  \Loglike^r(x_{1:n})\right)$.  We will be concerned with the geometry of the
level sets of $\psi$, i.e., the sets defined by $\psi^{-1}(c)$ for $c \in
\Reals^r$.  We say that a function $\psi:\ContinuousSpace^n\ra\mathbb{R}^r$
has {\bf has fibers with uniform bound $B$ on the number of path-components} if
$\psi^{-1}(x)$ has at most $B$ path-components, equivalence classes of points
where two points are equivalent if there is a path in $\psi^{-1}(x)$ connecting
them, for each $x\in\mathbb{R}^r$.

\begin{prop}
  Suppose that all functions in $\LikeFnClass$ are jointly continuous in their
  $d$ parameters almost everywhere, and that $\LikeFnClass$ has fibers with
  uniform bound $B$ on the number of path-components.  Then the growth function
  of $\LikeFnClass$, i.e., the maximum number of ways that $m \geq d$ data
  points $\adjacencymatrix^1, \ldots \adjacencymatrix^m$ could be dichotomized
  by thresholded functions from $\LikeFnClass$, is at most
  \begin{equation}
  \Pi(m) \leq B {\left(\frac{em}{d}\right)}^{d} \label{eqn:bounded-connected-components-growth-function}
  \end{equation}
  Thus the pseudo-dimension of $\LikeFnClass$ is at most
  $2\log_2{B}+2d\log_2{2/\ln{2}}$.
\label{prop:bounded-connected-components-growth-function}
\end{prop}

\begin{proof}
  The inequality \eqref{eqn:bounded-connected-components-growth-function} is a
  simplification of Theorem 7.6 of \citet[p.\
  91]{Anthony-Bartlett-neural-network-learning}, which allows for sets to be
  defined by $k$-term Boolean combinations of thresholded functions from
  $\LikeFnClass$.  (That is, the quoted bound is that of the theorem with
  $k=1$.)  Moreover, while Theorem 7.6 of
  \citet{Anthony-Bartlett-neural-network-learning} assumes that all functions
  in $\LikeFnClass$ are $C^d$, the proof ({\em op.\ cit.}, sec. 7.4) only
  requires continuity in the simplified setting $k=1$.

  For any class of sets with VC dimension $v < \infty$, the growth function is
  polynomial in $m$, $\Pi(m) \leq (em/v)^v$ \citep[Theorem 3.7, p.\
  40]{Anthony-Bartlett-neural-network-learning}, and, conversely, if
  $\Pi(m) < 2^m$ for any $m$, then the class of sets has VC dimension at most
  $m$.  Since Eq.\ \ref{eqn:bounded-connected-components-growth-function} shows
  that $\Pi(m)$ grows only polynomially in $m$, the VC dimension must be
  finite.  Comparing the $O((m/d)^d)$ rate of Eq.\
  \ref{eqn:bounded-connected-components-growth-function} to the $O((m/v)^v)$
  generic VC rate suggests $v=O(d)$, but it is desirable, for later purposes,
  to find a more exact result.

  To do so, we find the least $m$ where Eq.\
  \ref{eqn:bounded-connected-components-growth-function} is strictly below
  $2^m$, and take the logarithm:
  \begin{eqnarray}
    B {\left(\frac{em}{d}\right)}^{d} & < & 2^m\\
    \log_2{B} + d\log_2{\frac{e}{d}} + d\log_2{m} & < & m
  \end{eqnarray}
  Now, one can show that
  $\log_2{m} \leq \frac{m}{2d} + \log_2{\frac{2d}{e\ln{2}}}$ \citep[p.\
  91]{Anthony-Bartlett-neural-network-learning}, so that
  \begin{equation}
    \log_2{B} + d\log_2{\frac{e}{d}} + d\log_2{m} \leq \log_2{B} + d\log_2{\frac{e}{d}} + \frac{m}{2} + d\log_2{\frac{2d}{e\ln{2}}}
  \end{equation}
  and it will be sufficient for the right-hand side to be $< m$.  This in turn
  is implied by
  \begin{equation}
    2\log_2{B} + 2d\log_2{\frac{2}{\ln{2}}} <  m
  \end{equation}
  so this is an upper bound on the VC dimension of the subgraphs of
  $\LikeFnClass$, and so on the pseudo-dimension of $\LikeFnClass$.
\end{proof}

Next we bound the complexity of log-likelihoods for certain latent spaces.

\begin{thm}
  If $(\ContinuousSpace, \linkfnseq)$ is \regular{}, the pseudo-dimension of $\LikeFnClass$ is at most
  \begin{equation}
    2\log_2{B_{\LatentSpace}} + 2n\dim{\LatentSpace}\log_2{2/\ln{2}}, \label{eqn:pseudo-dimension-of-likelihood-functions}
  \end{equation}
  where $B_{\LatentSpace}$ is the number of path-components of the space of isometries on $M$.  
  $\isometries(\ContinuousSpace)$.
  \label{thm:pseudo-dimension-of-likelihood-functions}
\end{thm}
\begin{proof}
  By the fact that $(\ContinuousSpace, \linkfnseq)$ is \smooth{}, $\LikeFnClass$ is $C^{\infty}$ in all
  its $n\dim{\LatentSpace}$ continuous parameters, so in applying Proposition
  \ref{prop:bounded-connected-components-growth-function}, we may set
  $d=n\dim{\LatentSpace}$.  
  Define $\phi(x_{1:n};G)$ to be the function $M^n\ra\mathbb{R}^{n^2}$ sending a tuple $x_{1:n}$ to the vector whose $(pq)$th coordinate, for $1\leq p,q\leq n$, is $\MetricSymbol(x_{p},x_{q})$.  
  Define $T:\mathbb{R}^{n^2}\ra\mathbb{R}^n$ by the rule 
  $$T(y_1,y_2,\ldots,y_{n^2})=\frac{1}{n(n-1)}\sum_{1\leq p\leq q\leq n}y_{pq}.$$

  Note each $\NormalizedLoglike(-;G)\in\LikeFnClass$ satisfies $\NormalizedLoglike=T\phi(-;G)$.
  The preimage $T^{-1}(c)$ of a point under $T$, a linear transformation, is either empty or a (connected and convex) linear subspace of $\mathbb{R}^{n^2}$.  
  The function $\phi(-;G)$ has bounded connected components with bound $B_{\LatentSpace}$ because $\phi(x_{1:n})=\phi(y_{1:n})$ if and only if $\IsometryClass{x_{1:n}}=\IsometryClass{y_{1:n}}$ by $\infty$-homogeneity.
  Each $x_{1:n}\in M^n$ has a neighborhood $U$ (e.g. a product of normal convex neighborhoods of $x_1,x_2,\ldots,x_n$ in $M$) such that $\phi(U;G)$ is convex in $\mathbb{R}^{n^2}$.   
  It is then straightforward to show that every path in $\phi(M^n;G)$ starting from a point $\phi(x_{1:n};G)$ lifts under $\phi(-;G)$ to a path in $M^n$ from $x_{1:n}$.  
  Thus $CC(\phi(-;G)^{-1}(d))=CC((\phi(-;G)^{-1}(T^{-1}(c)))$ for each $d\in T^{-1}(c)$.  
  Thus 
  \begin{eqnarray*}
    CC(\phi(-;G)^{-1}(T^{-1}(c)))
    &\leq& CC(\phi(-;G)^{-1}(d))\quad d\in T^{-1}(c)\\
    &\leq& B_{\LatentSpace}.
  \end{eqnarray*}
  where $CC(X)$ denotes the number of path-components of a space $X$.
  Thus each $\Loglike(-;G)\in\LikeFnClass$ has bounded connected components with bound $B_{\LatentSpace}$.
  The hypotheses of Proposition \ref{prop:bounded-connected-components-growth-function} being
  satisfied, \eqref{eqn:pseudo-dimension-of-likelihood-functions} follows from
  Proposition \ref{prop:bounded-connected-components-growth-function}.
\end{proof}

\subsection{Pointwise Convergence of Log-Likelihoods}

\begin{lem}
  Suppose that all of the edges in $G$ are conditionally independent given some
  random variable $\mu$. Then for any $\epsilon>0$,
  \begin{equation}
    \label{eqn:first-concentration-bound}
    \Prob{|\Loglike(x_{1:n}) - \CrossEntropy(x_{1:n})| > \epsilon}\leq 
		2e^{{\left(-2\frac{n^2(n-1)^2\epsilon^2}{\sum_{p=1}^{n}{\sum_{q > p}{2\lambda^2_n(x_p, x_q)}}}\right)}}
  \end{equation}
  In particular, this holds when $G\sim\graph_n(\trueparam)$ or $G \sim \graph_n(f)$.
  \label{lemma:individual-concentration-at-individual-rate}
\end{lem}

\begin{proof}
  Changing a single $\adjacencymatrix_{pq}$, but leaving the rest the same,
  changes $\Loglike(x_{1:n}; \adjacencymatrix)$ by
	$\frac{1}{n(n-1)}\lambda_n(x_p,x_q)$.  The $\adjacencymatrix_{pq}$, for $p<q$, are all independent given $\mu$.  We may
  thus appeal to the bounded difference (McDiarmid) inequality \citep[Theorem
  6.2, p.\ 171]{Boucheron-Lugosi-Massart-concentration}: if $f$ is a
  function of independent random variables, and changing the $k^{\mathrm{th}}$
  variable changes $\Loglike$ by at most $c_k$, then
  \begin{equation}
    \Prob{|f-\Expect{f}|>\epsilon}\leq 2e^{\left(-\frac{\epsilon^2}{2\nu}\right)}
  \end{equation}
  where $\nu = \frac{1}{4}\sum{c_k^2}$.  In the present case, $c_{pq} =
  \lambda_n(x_p,x_q)$.  Thus,
  \begin{equation}
	  \nu = \frac{1}{4}\sum_{p=1}^{n}{\sum_{q > p}{n^{-2}(n-1)^{-2}\lambda^2_n(x_p,x_q)}}=\frac{1}{4n^2(n-1)^2}\sum_{p=1}^{n}\sum_{q > p}{\lambda^2_n(x_p, x_q)}
  \end{equation}
  and so $\Prob{|\Loglike(x_{1:n}) - \CrossEntropy(x_{1:n})| > \epsilon \mid \mu}$ is bounded from above by
  \begin{equation}
		2e^{\left(-\frac{2n^2(n-1)^2\epsilon^2}{\sum_{p=1}^{n}{\sum_{q>p}{\lambda^2_n(x_p,x_q)}}}\right)}
  \end{equation}
  Since the unconditional deviation probability 
  $$\Prob{|\Loglike(x_{1:n}) -
    \CrossEntropy(x_{1:n})| > \epsilon}$$
    is just the expected value of the
  conditional probability, which has the same upper bound regardless of $\mu$,
  the result follows (cf.\ \citealt[Theorem 2]{CRS-Leo-predictive-mixtures}).

  Finally, note that all edges in $\graph_n(\trueparam)$ are unconditionally independent, while those in
  $\graph_n(f)$ are conditionally independent given $X_{1:n}$, which plays the role of $\mu$.
\end{proof}

This lemma appears to give exponential concentration at an $O(n^4)$ rate, but
of course the denominator of the rate itself contains ${n \choose 2} = O(n^2)$
terms, so the over-all rate is only $O(n^2)$.  Of course, there must be some
control over the elements in the denominator.

\begin{lem}
  If $-\logitbound\leq\lambda_n(x_p,x_q)\leq\logitbound$, then for any $x_{1:n}$ and $\epsilon>0$,
  \begin{equation}
		\Prob{|\Loglike(x_{1:n}) - \CrossEntropy(x_{1:n})| > \epsilon}\leq 2 e^{\left(-2\frac{n(n-1)\epsilon^2}{\logitbound^2}\right)} \label{eqn:first-concentration-bound-at-uniform-rate}
  \end{equation}
  \label{lem:individual-concentration-at-uniform-rate}
\end{lem}

\begin{proof}
  By assumption, $\lambda^2_n(x_p,x_q) \leq \logitbound^2$. Thus
  $\sum_{p=1}^{n}{\sum_{q > p}{\lambda^2_n(x_p,x_q)}} \leq {n \choose
    2}\logitbound^2$, and the result follows from Lemma
  \ref{lemma:individual-concentration-at-individual-rate}.
\end{proof}

\subsection{Uniform Convergence of Log-Likelihoods}

Lemmas \ref{lemma:individual-concentration-at-individual-rate} and
\ref{lem:individual-concentration-at-uniform-rate} show that, with high
probability, $\Loglike(x_{1:n})$ is close to its expectation value
$\CrossEntropy(x_{1:n})$ for any given parameter vector $x_{1:n}$.  However, we
need to show that the MLE $\randomestimatedcoordinates$ has an {\em expected}
log-likelihood close to the optimal value.  We shall do this by showing that,
{\em uniformly} over $\LatentSpace^n$, $\Loglike(x_{1:n})$ is close to
$\CrossEntropy(x_{1:n})$ with high probability.  That is, we will show that
\begin{equation}
\sup_{x_{1:n}}{\left|\Loglike(x_{1:n}) - \CrossEntropy(x_{1:n})\right|} \convprob 0
\label{eqn:uniform-convergence-of-fluctuations}
\end{equation}
This is a stronger conclusion than even that of Lemma
\ref{lem:individual-concentration-at-uniform-rate}: since $\LatentSpace$ is a
continuous space, even if each parameter vector has a likelihood which is
exponentially close to its expected value, there are an uncountable infinity of
parameter vectors.  Thus, for all we know right now, an uncountable infinity of
them might be simultaneously showing large deviations, and continue to do so no
matter how much data we have.  We will thus need to show that likelihood at
different parameter values are {\em not} allowed to fluctuate independently, but
rather are mutually constraining, and so eventually force uniform convergence.

If there were only a finite number of allowed parameter vectors, we could
combine Lemma \ref{lem:individual-concentration-at-uniform-rate} with a union
bound to deduce \eqref{eqn:uniform-convergence-of-fluctuations}.  With an
infinite space, we need to bound the covering number of $\LikeFnClass$.  To
recall\footnote{See, e.g., \citet{Anthony-Bartlett-neural-network-learning} or
  \citet{Vidyasagar-on-learning-and-generalization}.}, the $L_1$ covering
number of a class $F$ of functions at scale $\epsilon$ and $m$ points,
$\mathcal{N}_1(\epsilon, F, m)$, is the cardinality of the smallest set of
functions $f_j \in F$ which will guarantee that, for any choice of points
$a_1, \ldots a_m$, $\sup_{_1,\ldots,a_m}\frac{1}{m}\sum_{i=1}^{m}{|f(a_i)f_j(a_i)|} \leq \epsilon$ for some
$f_j$ (this definition can be straightforwardly shown to be equivalent to that of \citet{Anthony-Bartlett-neural-network-learning}).  Typically, as in \citet[Theorem 17.1, p.\
241]{Anthony-Bartlett-neural-network-learning}, a uniform concentration
inequality takes the form of
\begin{equation}
\Prob{\sup_{f \in F}{\left|f - \Expect{f}\right|} \geq \epsilon} \leq c_0 c_1 \mathcal{N}_1(\epsilon c_2, F, c_3 m) e^{\left( - c_4 \epsilon^2 r(m)\right)}
\end{equation}
where the individual deviation inequality is
\begin{equation}
	\Prob{|f - \Expect{f}| \geq \epsilon} \leq c_0 e^{-\left( - \epsilon^2 r(m) \right)}.
\end{equation}

In turn, \citet[Theorem 18.4, p.\
251]{Anthony-Bartlett-neural-network-learning} shows that the $L_1$ covering
number $\mathcal{N}_1(\epsilon, F, m)$ of a class $F$ of functions with finite pseudo-dimension $v$ at scale
$\epsilon$ and $m$ observations is bounded:
\begin{equation}
	\mathcal{N}_1(\epsilon, F, m) \leq e(v+1){\left(\frac{2e}{\epsilon}\right)}^v.
\end{equation}
In our setting, we have $m=1$.  (That is, we observe {\em one} high-dimensional
sample; notice that the bound is independent of $m$ so this hardly matters.)

It thus remains to bound the pseudo-dimension of $\LikeFnClass$.  This involves
a rather technical geometric argument, ultimately revolving on the group
structure of the isometries of $\MetricSpace$. This may be summed up in the
existence of a constant $B_{\ContinuousSpace}$, which is 2 for any Euclidean
space, and (as it happens) also 2 for $\HH$.  This matter was handled in \S
\ref{sec:dimension-bound-details}.

\subsection{Proof of Theorem \ref{thm:uniform-concentration-of-likelihood}}

By assumption, there exists a sequence $\nu_1,\nu_2,\ldots$ of non-negative reals such that $|\lambda_n(x_p,x_q)|\leq\logitbound$ for each $n$ and $p,q$ with $\nu_n\in o(\sqrt{n})$.  

Presume for the moment that we know the $L_1$ covering number of
$\LikeFnClass$ is at most $\mathcal{N}_1(\LikeFnClass,\epsilon,1)$.  Then
  \begin{equation}
		\Prob{\sup_{x_{1:n}}{|\Loglike(x_{1:n})-\CrossEntropy(x_{1:n})|} \geq \epsilon} \leq 4\mathcal{N}_1(\LikeFnClass, \epsilon/16,2)e^{\left(-\frac{\epsilon^2n(n-1)}{8\logitbound^2}\right)}
    \label{eqn:concentration-bound-in-terms-of-covering-number}
  \end{equation}
The proof is entirely parallel to that of Theorem 17.1 in \citet[p.\
 241]{Anthony-Bartlett-neural-network-learning}, except for using Lemma
\ref{lem:individual-concentration-at-uniform-rate} in place of Hoeffding's
inequality, and so omitted.

Now, by Proposition \ref{prop:reasonable.symmetries} $B_M=2$ and therefore by
Theorem \ref{thm:pseudo-dimension-of-likelihood-functions}, the
pseudo-dimension of $\LikeFnClass$ is at most $2\log_2{B_{\LatentSpace}} +
2n\dim{\LatentSpace}\log_2{2/\ln{2}}$.  The $L_1$ covering number of
$\LikeFnClass$ is thus exponentially bounded in $O(n\log{1/\epsilon})$,
specifically \citep[Theorem 18.4, p.\
251]{Anthony-Bartlett-neural-network-learning}:
$\mathcal{N}_1(\LikeFnClass,\epsilon,2)$ is bounded above by
  \begin{equation}
    e(1 + 2\log_2{B_{\LatentSpace}} + 2n\dim{\LatentSpace}\log_2{2/\ln{2}}){\left(\frac{2e}{\epsilon}\right)}^{2\log_2{B_{\LatentSpace}} + 2n\dim{\LatentSpace}\log_2{2/\ln{2}}} 
    \label{eqn:covering-number-of-rigid-CLS}
\end{equation}

\eqref{eqn:covering-number-of-rigid-CLS} grows exponentially in
$O(n\log{1/\epsilon})$, while the rightmost factor in the upper bound of
\eqref{eqn:concentration-bound-in-terms-of-covering-number} shrinks
exponentially in $O(\epsilon^2n^2/\logitbound^2)$ and hence $O(n\epsilon^2)$ by
our regularity assumption.  For fixed $\epsilon$, then, the uniform deviation
probability over all of $\LikeFnClass$ in
\eqref{eqn:concentration-bound-in-terms-of-covering-number} is therefore
exponentially small, hence we have convergence in probability to zero.  $\Box$

{\em Remark 1:} In applying the theorems from
\citet{Anthony-Bartlett-neural-network-learning}, remember that we have only
one sample ($m=1$), which is however of growing ($O(n^2)$) dimensions, with a
more-slowly growing ($O(n)$) number of parameters.

{\em Remark 2:} From the proof of the theorem, we see that if $\logitbound^2$ grows slowly enough, the sum of the deviation
probabilities tends to a finite limit.  Convergence in probability would then
be converted to almost-sure convergence by means of the Borel-Cantelli lemma,
{\em if} the graphs at different $n$ can all be placed into a common
probability space.  Doing so however raises some subtle issues we prefer not to
address here (cf.\ \citep{your-favorite-ergm-sucks}).

\subsection{Proof of Corollary \ref{cor:uniform-concentration-of-likelihood}}

We adapt a very standard pattern of argument used to prove oracle inequalities
in learning theory.  This begins with Lemma
\ref{lemma:crossent-minimized-on-truth}, that $\CrossEntropy(\trueparam) \geq
\CrossEntropy(\MLE)$.  This implies that
$|\CrossEntropy(\MLE) - \CrossEntropy(\trueparam)| =
\CrossEntropy(\trueparam) - \CrossEntropy(\MLE)$.  Now add
and subtract log-likelihoods:
\begin{eqnarray}
0 \leq \CrossEntropy(\trueparam) - \CrossEntropy(\MLE) & = & \CrossEntropy(\trueparam) -\Loglike(\MLE) + \Loglike(\MLE) - \CrossEntropy(\MLE)\\
& \leq & \CrossEntropy(\trueparam) -\Loglike(\trueparam) + \Loglike(\MLE) - \CrossEntropy(\MLE) \label{eqn:mle-has-best-loglike} \\
& \leq & |\CrossEntropy(\trueparam) -\Loglike(\trueparam)| + |\Loglike(\MLE) - \CrossEntropy(\MLE)|\\
& \leq & 2 \sup_{x_{1:n}}{|\Loglike(x_{1:n})-\CrossEntropy(x_{1:n})|} \convprob 0
\end{eqnarray}
where in Eq.\ \ref{eqn:mle-has-best-loglike} we use the trivial fact that since
$\MLE$ maximizes the likelihood, $\Loglike(\trueparam) \leq
\Loglike(\MLE)$, and the last line invokes Theorem
\ref{thm:uniform-concentration-of-likelihood}.  $\Box$

\section{Conclusion}
We have formulated and proven a notion of convergence for non-parametric likelihood estimators of graphs generated from continuous latent space models, under some mild assumptions on the generative models.
Traditional convergence results for statistical estimators are a kind of ergodicity, or long-term mixing, for multiple, independent samples. 
The size of a single sample network here plays the role of the number of samples in traditional formulations of consistency.
These main results hold even when our generative models are mis-specified, i.e. when we fix a latent space but the generating graph distributions are not defined in terms of the space, under some additional assumptions [Appendix \ref{appendix:mis-specified}].
Continuous latent space models turn out to provide the necessary ergodicity through conditional independence.
A consequent notion of consistency, which we save for future work, requires some formalization of what we mean by convergence of estimates, i.e. sequences of coordinates, of varying sizes.  
And a proof of such a consistency result will likely require some adaptation of standard technical tools for concluding convergence of extremal estimators from convergence in random objective functions (eg. \cite{van-der-Vaart-asymptotic-stats}.)

\appendix
\section{Mis-specified models}
\label{appendix:mis-specified}
Our consistency results extend from specified to certain mis-specified models.
We still assume the existence of a latent space $(\ContinuousSpace,\linkfnseq)$
as before, but assume that sample graphs are sampled not by a distribution of
the form $\graph_n(x_{1:n})$ but in fact by some arbitrary distribution of
graphs having $n$ nodes.
The only assumption we make about such random graphs $G$ in this section, as before, is that there exists some random variable $\mu$ such that the edges of $G$ are conditionally inependent given $\mu$.
For the case where $G$ is drawn from a CLS model, $\mu$ can be taken to be the random latent coordinates of the nodes of $G$.
We call a sequence $G_1,G_2,\ldots$ of random graphs
{\em almost-specified} if there exists $x^*_{1:\infty}\in\ContinuousSpace^\infty$
such that, for all sufficiently large $n$, $\CrossEntropy(x_{1:n})$ achieves a
maximum uniquely exactly for $x_{1:n}\in\IsometryClass{\trueparam}$.  For such
an almost-specified model, $x^*_{1:\infty}$ plays the role of the true coordinates
and the assumption of being almost specified plays the role of Lemma
\ref{lemma:crossent-minimized-on-truth} (e.g. in all proofs); we call such
$x^*_{1:\infty}$ the {\em pseudo-coordinates} of the almost-specified model.
Consequently, we can restate our main results at the following level of
generality.

\begin{thm}
  For an almost specified model with pseudo-coordinates $x^*_{1:*}$ and a compact, \regular{} latent space $(\ContinuousSpace, \linkfnseq)$,
  \begin{equation*}
    \sup_{x_{1:n}}{|\Loglike(x_{1:n})-\CrossEntropy(x_{1:n})|} \convprob 0
  \end{equation*}
\end{thm}


\section*{Acknowledgements}

Our work was supported by NSF grant DMS-1418124; DA also received support from
NSF Graduate Research Fellowship under grant DGE-1252522, and CRS from NSF
grant DMS-1207759.  We are grateful for valuable discussions with Carl
Bergstrom, Elizabeth Casman, David Choi, Aaron Clauset, Steve Fienberg,
Christopher Genovese, Aryeh Kontorovich, Dmitri Krioukov, Cris Moore,
Alessandro Rinaldo, Mitch Small, Neil Spencer, Andrew Thomas, Larry Wasserman,
and Chris Wiggins, and for feedback from seminar audiences at CMU, UCLA,
UW-Seattle, and SFI.

\bibliography{locusts}
\bibliographystyle{imsart-nameyear}

\end{document}